\documentclass{amsart}

\usepackage{amssymb}
\usepackage{mathrsfs}
\usepackage[colorlinks=true,linkcolor=blue,citecolor=magenta]{hyperref}

\newtheorem{thm}{Theorem}[section]

\newtheorem{lem}[thm]{Lemma}
\newtheorem*{lem*}{Lemma}

\newtheorem{prop}[thm]{Proposition}

\newtheorem{cor}[thm]{Corollary}

\newtheorem{question}[thm]{Question}

\theoremstyle{remark}
\newtheorem{remark}[thm]{Remark}

\theoremstyle{definition}
\newtheorem{defn}[thm]{Definition}

\newtheorem{example}[thm]{Example}

\newcounter{my_enumerate_counter}

\newcommand\comment[1]{}

\newcommand\Homeo{\operatorname{Homeo}}

\renewcommand{\epsilon}{\varepsilon}

\title[]
{Coherent actions by homeomorphisms on the real line or an interval}

\thanks{
The author thanks Nicol\'{a}s Matte Bon, Matt Brin, Mark Sapir, Justin Moore, Isabelle Liousse and Michele Triestino for helpful discussions and comments.
This research has been supported by a Swiss national science foundation ``Ambizione" grant.}

\author{Yash Lodha}



\begin{document}

\dedicatory{This paper is dedicated to the memory of Prof. Matti Rubin (1946-2017).} 

\begin{abstract}
We study actions of groups by homeomorphisms on $\mathbf{R}$ (or an interval) that are minimal, have solvable germs at $\pm \infty$ and contain a pair of elements of a certain type.
We call such actions \emph{coherent}.
We establish that such an action is rigid, i.e. any two such actions of the same group are topologically conjugate.
We also establish that the underlying group is always non elementary amenable, but satisfies that every proper quotient is solvable.
As a first application, we demonstrate that any coherent group action 
$G<\textup{Homeo}^+(\mathbf{R})$ that produces a nonamenable equivalence relation with respect to the 
Lebesgue measure satisfies that the underlying group does not embed into Thompson's group $F$.
This includes all known examples of nonamenable groups that do not contain non abelian free subgroups and act faithfully on the real line by homeomorphisms.
As a second application, we establish that the Brown-Stein-Thompson groups $F(2,p_1,...,p_n)$ for $n\geq 1$ and $p_1,...,p_n$ distinct odd primes, do not embed into Thompson's group $F$.
This answers a question recently raised by C. Bleak, M. Brin, and J. Moore.
Our tools also allow us to prove additional non embeddability results for Brown-Stein-Thompson and Bieri-Strebel groups.
\end{abstract}

\maketitle

\section{Introduction}

We define a group action $G<\textup{Homeo}^+(\mathbf{R})$ to be \emph{coherent} if:
\begin{enumerate}
\item The action is minimal, i.e. the orbits are dense.
\item The groups of germs at $\pm\infty$ are solvable.
\item There exists an element that has a trivial germ at $-\infty$ and does not fix any point in some interval $(r,\infty)$.
\item There exists an element that has a trivial germ at $+\infty$ and does not fix any point in some interval $(-\infty, s)$.
\end{enumerate}
(A similar definition is prescribed for a group action  $G<\textup{Homeo}^+([0,1])$, in the preliminaries.
Also, groups of germs are defined in Definition \ref{GOgerms}.)

These conditions are satisfied by a rich class of group actions by homeomorphisms, many of which are discussed in subsection \ref{Examples}. 
The class of groups that admit such actions is denoted by $\mathcal{C}$.
The class contains continuum many isomorphism classes of finitely generated groups (see Lemma \ref{Continuum}), and any group that admits a faithful action on the real line by homeomorphisms embeds in some group in this class.
(A broad range of examples are described in \ref{Examples}.)
Groups in $\mathcal{C}$ have interesting algebraic and dynamical features, and we show the following:

\begin{thm}\label{structure}
Let $G\in \mathcal{C}$. Then $G$ satisfies the following:
\begin{enumerate}
\item $G$ contains a subgroup isomorphic to Thompson's group $F$.
Therefore $G$ is non elementary amenable (in particular, $G$ is non solvable).
\item There exists an $n\in \mathbf{N}$ (which depends on $G$) such that every proper quotient of $G$ is solvable of degree at most $n$.
\item There exists an $n\in \mathbf{N}$ (which depends on $G$) such that the $n$-th derived subgroup $G^{(n)}$ is simple.
\end{enumerate}
\end{thm}

We remark that for each $n\in \mathbf{N}$, there is a group in $\mathcal{C}$ which admits a quotient that is solvable of degree $n$.
(See Lemma \ref{solvability}.)
We demonstrate that coherent group actions are \emph{rigid}:

\begin{thm}\label{rigidity}
Consider two coherent actions $G,H<\textup{Homeo}^+(\mathbf{R})$ such that the underlying groups $G,H$ are isomorphic.
Then for each isomorphism $\nu:G\to H$ there is a homeomorphism $\phi:\mathbf{R}\to \mathbf{R}$ such that $\nu(f)=\phi^{-1} f \phi$ for each $f\in G$.
\end{thm}

We also obtain the following related statement as a direct corollary of the proofs.
Here we can replace the conditions in the definition of coherent actions with milder hypothesis.

\begin{thm}
Let $G,H<\textup{Homeo}^+(\mathbf{R})$ be group actions such that:
\begin{enumerate}
\item Both actions are minimal.
\item Both actions contain non identity elements with supports contained in compact intervals.
\item There exist $g\in G, h\in H$ such that $g,h$ have trivial germs at $-\infty$ and there is an interval $(r,\infty)$
on which both $g,h$ are increasing.
\item $G\cong H$.
\end{enumerate}
Then for each isomorphism $\nu:G\to H$ there is a homeomorphism $\phi:\mathbf{R}\to \mathbf{R}$ such that $\nu(f)=\phi^{-1} f \phi$ for each $f\in G$.
\end{thm}



Groups in the family $\mathcal{C}$ are interesting objects of study for their own sake, however in this article we are also driven by the following applications.
We shall prove two non-embeddability results for Thompson's group $F$.
Along the way, we develop a systematic set of tools that can be employed for proving non-embeddability results for various groups in the Thompson family. 

Thompson's group $F$ was introduced by R. Thompson more than $50$ years ago, and the group is an interesting and important infinite group that has 
been postulated as a counterexample for various conjectures in group theory (see \cite{CFP} for instance.)
The subgroup structure of $F$ is quite mysterious and several recent papers (for instance \cite{BBKMZ}, \cite{BBM}, \cite{GS} and \cite{Golan}) have been 
devoted to developing a systematic structure theory of subgroups of $F$.
One prominent result is  due to Vaughan Jones (see \cite{Jones}), who showed that the the group $F$ encodes all knots in a natural manner, and a specific subgroup  of $F$ encodes all oriented knots.

The question concerning the amenability of Thompson's group $F$ is well known, and the group is lurking around the boundary between amenable and nonamenable groups.
In \cite{Monod}, Monod employed a remarkable strategy to establish non amenability of various groups of homeomorphisms of the real line.
The strategy involves establishing non amenability of the group by demonstrating the non $\mu$-amenability of the associated orbit equivalence relation (with respect to the Lebesgue measure $\mu$).
In this light, it is natural to inquire the following:\\

\begin{question}\label{Q1}
Let $G<\textup{Homeo}^+(\mathbf{R})$ be a group action such that the associated orbit equivalence relation is non $\mu$-amenable.
Does the underlying group $G$ embed in Thompson's group $F$?
\end{question}

We obtain a negative result to the above question for the groups in class $\mathcal{C}$ whose coherent actions produce non $\mu$-amenable equivalence relations.

\begin{thm}\label{main2}
Let $G<\textup{Homeo}^+(\mathbf{R})$ be a coherent group action which produces a non $\mu$-amenable equivalence relation (with respect to the Lebesgue measure).
Then the underlying group $G$ does not embed in Thompson's group $F$.
\end{thm}

The second result concerns a question recently raised by C. Bleak, M. Brin, and J. Moore that asks whether Brown-Stein-Thompson groups embed in $F$.
We show that this is never the case, and in fact we obtain the following more general statement.

\begin{thm}\label{main1}
Let $G<\textup{Homeo}^+(\mathbf{R})$ be a coherent action of a finitely generated group for which there exists an $x\in \mathbf{R}\cup \{\pm \infty\}$ so that the group(s) of germs at $x$ (of the elements that fix $x$)
is either non abelian, or abelian of rank greater than $1$.
Then the underlying group $G$ does not embed into Thompson's group $F$.
\end{thm}

Since the standard actions of Brown-Stein-Thompson groups are coherent actions that satisfy the above hypothesis on groups of germs, we obtain the following.
 
\begin{cor}\label{maincor1}
Let $G=F(2,p_1,...,p_n)$ for $n\geq 1$ and $p_1,...,p_n$ distinct odd primes be a Brown-Stein-Thompson subgroup of $\textup{PL}^+([0,1])$.
Then $G$ does not embed into Thompson's group $F$.
The Bieri-Strebel group $G(I; A, P)$ does not embed into $F$ provided $P<\mathbf{R}_+^*$ has abelian rank greater than one.
\end{cor}

In fact we are also able to conclude the following more general statement.

\begin{cor}\label{maincor2}
Let $G=F(p_1,...,p_n)$ and $H=F(q_1,...,q_m)$ be  Brown-Stein-Thompson groups where $(p_1,...,p_n)$ and $(q_1,...,q_m)$ are tuples of distinct primes and $m,n\geq 2$.
If $m>n$, then $H$ does not embed in $G$.
The Bieri-Strebel group $G(I; A_1, P_1)$ does not embed into the Bieri-Strebel group $G(I; A_2, P_2)$ if the abelian rank of $P_1$ is greater than that of $P_2$.
\end{cor}

We also show the following general obstruction to embeddability into $F$.

\begin{thm}\label{main3}
Let $G<\textup{Homeo}^+(\mathbf{R})$ be a coherent action of a finitely generated group such that there exists an element $g\in G$ with infinitely many components of support.
Then the underlying group $G$ does not embed into Thompson's group $F$.
\end{thm}

As another corollary to Theorem \ref{main1} we immediately obtain:

\begin{cor}\label{maincor3}
Let $G$ be a subgroup of Monod's group $H(\mathbf{R})$ acting coherently and so that there is a point $x\in \mathbf{R}\cup \{\pm \infty\}$ such that the group(s) of germs at $x$ is metabelian but not abelian, or abelian of rank greater than $1$.
Then $G$ does not embed into Thompson's group $F$.
\end{cor}

\section{Preliminaries}
All actions will be right actions, unless otherwise specified.
Let $I=\mathbf{R}$ or a compact subinterval.
Given an element $g\in \textup{Homeo}^+(I)$, the \emph{support} of $g$, or $supp(g)$, is defined as the set $$\{x\in I\mid x\cdot g\neq x\}$$
Given an interval $J$ in the domain of $g$, $g$ is said to be increasing on $J$ if for each $x\in J$, $x<x\cdot g$.

\begin{defn}\label{GOgerms}
Let $G<\textup{Homeo}^+(I)$.
Given an $x\in I$, $G_x$ is defined to be the subgroup of $G$ that fixes $x$.
If $x\neq sup(I)$, the group $G_{x}^+$ is the subgroup consisting of the elements:
$$\{g\in G\mid \exists \delta>0\text{ such that }g\text{ fixes each point in }[x,x+\delta]\}$$
If $x\neq inf(I)$, the group $G_{x}^-$ is the subgroup consisting of the elements:
$$\{g\in G\mid \exists \delta>0\text{ such that }g\text{ fixes each point in }[x-\delta,x]\}$$
We remark that in the above definitions the $\delta$ in the definition depends on the element $g$.

For $G<\textup{Homeo}^+(I)$ the \emph{groups of germs}  for any $x\in I\cup \{sup(I),inf(I)\}$ are defined to be the groups
the groups $G_x/G_{x}^+$ and $G_x/G_{x}^-$.
The family of \emph{groups of germs at fixed points} of $G$ is defined to be the family $$\{G_x/G_x^+, G_x/G_x^-\mid x\in I\cup \{sup(I),inf(I)\}\}$$  
\end{defn}

Note that if $I=\mathbf{R}$, the we define the groups of germs at $\pm \infty$ analogously.
So for instance if $x=\infty$, then we have $G_{\infty}=G$ and $$G_{\infty}^-=\{g\in G\mid \exists r\in \mathbf{R}, g\text{ fixes each point in }(r,\infty)\}$$
So the group of germs at $\infty$ is $G_{\infty}/G_{\infty}^-$.

Coherent actions are defined in the introduction.
A similar definition is applicable for a group action $G<\textup{Homeo}^+(I)$ for a compact interval $I$,
which for the sake of completeness we mention below:
\begin{enumerate}
\item The action is minimal on $int(I)$, i.e. the orbits are dense in $int(I)$.
\item The groups of germs at $inf(I),sup(I)$ are solvable.
\item There is an element that has a trivial germ at $inf(I)$ and does not fix any point in an interval of the form $(r,sup(I))$ for some $r\in int(I)$.
\item There is an element that has a trivial germ at $sup(I)$ and does not fix any point in an interval of the form $(inf(I),s)$ for some $s\in int(I)$.
\end{enumerate}

Note that coherent actions of a group on a (or any) compact interval and on the real line are always topologically conjugate.
Therefore the choice of $I$ will not matter in this article.
We shall mostly follow the setup that is most convenient to the particular actions we are studying.

\subsection{Examples of coherent actions}\label{Examples}
 
There are many natural examples of coherent group actions,
such as various groups of piecewise linear and piecewise projective homeomorphisms.
We describe two such actions of Thompson's group $F$ that will feature prominently in this article:

\begin{example}\label{F1}
Thompson's group $F$ is isomorphic to the group of piecewise $PSL_2(\mathbf{Z})$ homeomorphisms of the real line with breakpoints in the set $\mathbf{Q}$.
\end{example}

\begin{example}\label{F2}
Thompson's group $F$ is isomorphic to the group of piecewise linear homeomorphisms of $[0,1]$ satisfying that:
\begin{enumerate}
\item The derivatives, wherever they exist, are powers of $2$.
\item The breakpoints, i.e. the points where derivatives do not exist, are dyadic rationals.
\end{enumerate}
\end{example}

The following are also prominent examples of actions that satisfy the conditions of coherence.

\begin{example}
(Brown-Stein-Thompson groups) Let $p_0,...,p_n$ be distinct odd primes.
The group $F(2,p_0,...,p_n)$ is the group of piecewise linear homeomorphisms of $[0,1]$ such that:
\begin{enumerate}
\item The slopes, wherever they exist, lie in $\langle 2,p_0,...,p_n\rangle <\mathbf{R}_+^*$.
\item The breakpoints lie in the set $\mathbf{Z}[\frac{1}{2p_0...p_n}]\cap (0,1)$.
\end{enumerate}
\end{example}

\begin{example}
(Bieri-Strebel groups)  The group $G(I; A, P)$ consists of all orientation preserving piecewise linear homeomorphisms of the real line with support in the interval $I$, 
slopes in a multiplicative subgroup $P$ of the positive reals and breaks in a the additive $\mathbf{Z}[P]$-submodule $A$ of $\mathbf{R}^+$ (that lie in $I$).
\end{example}

\begin{example}
(Monod's groups) Let $A<\mathbf{R}$ be a subring. Let $Q_A$ be the set of fixed points of hyperbolic elements of $PSL_2(A)$.
The group $H(A)$ is the group of piecewise $PSL_2(A)$ homeomorphisms of $\mathbf{R}$ with breakpoints in $Q_A$.
These groups were shown in \cite{Monod} to be nonamenable despite the fact that they don't contain non abelian free subgroups.
\end{example}

\begin{example}
Any overgroup of $F$ (from example \ref{F1}) in $H(\mathbf{R})$ will be a coherent action.
These groups are often nonamenable (with no non abelian free subgroups) and admit paradoxical decompositions with $25$ pieces (see \cite{Lodha} for instance), and are test cases for Question \ref{Q1}.
\end{example}

\begin{example}
(A finitely presented nonamenable example) Let $G$ be the group generated by 
$$
a(t)=t+1\qquad b(t)=\left\{\begin{array}{ll}
\smallskip
t&\text{ if }t\leq 0\\
\medskip
\dfrac{t}{1-t}&\text{ if }0\leq t\leq \dfrac12\\
\medskip
\dfrac{3t-1}{t}&\text{ if }\dfrac12\leq t\leq 1\\
t+1&\text{ if }1\leq t\end{array}\right.
\smallskip
c(t)=
\begin{cases}
t&\text{ if }t\leq 0\\
 \frac{2t}{1+t}&\text{ if }0\leq t\leq 1\\
t&\text{ if }t\geq 1\\
\end{cases}
$$
It was shown in \cite{LodhaMoore} that the group $G$ is finitely presented, nonamenable, and does not contain free subgroups.
\end{example}

\begin{example}
(The broken Baumslag-Solitar groups)
For each $\lambda\in \mathbf{Q}_{>0}$ we define the group $G_{\lambda}<\textup{Homeo}^+(\mathbf{R})$ as generated by $a(t)=t+1$
together with $$b_+(t)=t\text{ if }t\geq 0\qquad b_+(t)=\lambda t\text{ if }t\leq 0$$ 
and
$$b_-(t)=t\text{ if }t\leq 0\qquad b_-(t)=\lambda t\text{ if }t\geq 0$$
It is easy to see that $G_{\lambda}$ contains the affine group generated by $t\to t+1,t\to \lambda t$,
which acts minimally, and hence $G_{\lambda}$ is minimal.
The groups of germs at $\pm \infty$ are metabelian, and the generators $b_+,b_-$
are the required elements in the definition of Coherent actions.
It follows that the actions of the groups $G_{\lambda}$ are coherent.
In \cite{BonnatiLodhaTriestino} it was shown that these groups do not admit faithful $C^1$-actions on a closed interval.
\end{example}

\begin{example}
(Minimal pre-chain groups)  We let $n\geq 2$ and let $\mathcal{J}=\{J_1,\ldots,J_n \}$ 
be a chain of open intervals such that $inf(J_{i+1})\in J_i$ and $sup(J_{i+1})>sup(J_i)$ for each $i<n$.
We consider a collection of homeomorphisms $\{f_1,\ldots,f_n \}$ such that  $supp(f_i)=J_i$
and such that $f_i(t)>t$ for each $t\in int(J_i)$.
We set $G=\langle f_1,...,f_n\rangle <\Homeo^+(\mathbf{R})$.
Such a group is called a \emph{pre-chain group}.
Such groups were studied extensively in \cite{KKL}.
A pre-chain group acting minimally is a coherent action.
\end{example}

We remark that in the above examples which involve ``piecewise'' constructions, each element is only allowed to have finitely many breakpoints.
However, one may consider analogous ``piecewise'' coherent actions where the number of breakpoints is allowed to be infinite.
In general we may consider groups with very complicated dynamical descriptions.
For instance, the following examples arise from groups with trivial or abelian germs at $\pm \infty$.

\begin{example}
Let $\textup{Homeo}^+_C(\mathbf{R})$ be the group of homeomorphisms generated by elements of $\textup{Homeo}^+(\mathbf{R})$
which satisfy that their supports are contained in compact intervals.
Let $\textup{Homeo}^+_T(\mathbf{R})$ be the group of homeomorphisms generated by elements of $\textup{Homeo}^+(\mathbf{R})$
which satisfy that their germs at $\pm \infty$ are (possibly trivial) translations.
Then for any subgroup $H<\textup{Homeo}^+_C(\mathbf{R})$ or $H<\textup{Homeo}^+_T(\mathbf{R})$, the group $G$ generated by $H$ and $F$ (as in example \ref{F1}) is coherent.
\end{example}

More generally, we can construct classes of examples with solvable germs at $\pm \infty$ of desired length.
In the proof of Lemma \ref{solvability}, we construct for each $n\in \mathbf{N}$ a coherent action whose underlying group admits a solvable quotient of length $n$.

Given a finitely generated group $G<\textup{Homeo}^+(I)$, an open interval $J\subset I$ is said to be an \emph{orbital},
if it is $G$-invariant and there is no proper invariant subinterval which is $G$-invariant.
Equivalently, $I$ is a connected component of the support of $G$.
The following is elementary.

\begin{lem}\label{orbital}
Given a group $G<\textup{Homeo}^+(I)$, there is a family of closed intervals $I_0,I_1,...$ in $I$
such that:
\begin{enumerate}
\item $int(I_n)\cap int(I_m)=\emptyset$ if $n\neq m$ and $\bigcup_{n\in \mathbf{N}} I_n=I$.
\item Each $I_n$ is either an orbital for $G$ or $G$ fixes each point in $I_n$. 
\end{enumerate}
\end{lem}


Given a group $G<\textup{Homeo}^+(I)$, a closed $G$-invariant set $J\subset I$ is said to be a \emph{minimal invariant set} if the action of $G$ restricted to $J$ is minimal, i.e. has dense orbits.
We say that $J$ is \emph{exceptional} if $J$ is perfect and totally disconnected.
The following holds (See \cite{Navas}, section $2.1.2$).

\begin{lem}\label{minimalinvariant}
Let $G<\textup{Homeo}^+(I)$ be an action of a finitely generated group such that there is no proper invariant subinterval of $I$.
(In particular, there are no global fixed points in $int(I)$.)
Then either of the following holds:
\begin{enumerate}
\item The action of $G$ on $int(I)$ is minimal.

\item There is a discrete subset $\Gamma\subset int(I)$ such that $\Gamma$ is $G$-invariant.

\item There is an exceptional, minimal $G$-invariant set $\Gamma\subset int(I)$. 
Moreover, if there is an $x\in int(I)$ such that $x$ is an accumulation point of each orbit of $G$,
then such a set is unique and equals $\overline{x\cdot G}$.

\end{enumerate}
\end{lem}

The following is a general dichotomy for $1$-dimensional homeomorphism
groups, and the proofs are variations of \cite{Navas}, which we omit.

\begin{lem}\label{DynDichotomy}
Let $G<\textup{Homeo}^+(I)$ such that $G$ admits a unique exceptional, minimal, invariant set $\Gamma\subset I$.
Then there is a map $\Phi:I\to I$, a group $H<\textup{Homeo}^+(I)$ and a homomorphism $\psi:G\to H$,
such that:
\begin{enumerate}
\item $H$ is isomorphic to $G/K$ where $K$ is the kernel of the restriction of the action of $G$ to $\Gamma$.
\item For each $g\in G$, $\Phi\circ g=\psi(g)\circ \Phi$.
\item The action of $H$ on $int(I)$ is minimal.
\end{enumerate}
\end{lem}

A tool we shall require is the following theorem, due to Higman (\cite{Higman}).
Let $\Gamma$ be a group of bijections of some set $E$. 
For $g\in \Gamma$ define its \emph{moved set} $D(g)$ as the set of points $x\in E$ such that $g(x)\ne x$. This is analogous to the support, but since \emph{a priori} there is no topology on $E$, we do not take the closure.

\begin{thm}\label{Higman}
(Higman's simplicity criterion) Suppose that for all $\alpha, \beta, \gamma\in \Gamma\setminus \{1_{\Gamma}\}$, there is a $\rho\in \Gamma$ such that:
$$\gamma(\rho(S))\cap \rho(S)=\varnothing\qquad  \text{where } S=D(\alpha)\cup D(\beta)$$
Then $\Gamma'$ is simple.
\end{thm}

$G<\textup{Homeo}^+(\mathbf{R})$ is said to be \emph{locally minimal} if for each triple $(U,V,x)$ where $U\subset V$ are open intervals and $x\in V$,
there is an element $g\in G$ such that:
\begin{enumerate}
\item $g$ fixes $I\setminus V$ pointwise.
\item $x\cdot g\in U$
\end{enumerate}

The main tool we shall require to prove rigidity of coherent actions is Rubin's theorem (see Section $9$ in \cite{Brin}).
We need the following terminology before we state the Theorem.
Let $X$ be a locally compact, Hausdorff space with no isolated points.
For any open set $V\subset X$, $G_V$ is defined as the subgroup of $G$ that consists of elements that pointwise fix $X\setminus V$.
$G<\textup{Homeo}^+(X)$ is said to be \emph{locally dense} if for each pair $(V,x)$ where $V$ is an open set and $x\in V$,
$\overline{x\cdot G_V}$ has non empty interior.

\begin{thm}\label{rubin}
(Rubin's theorem) Let $X$ be a locally compact, Hausdorff space with no isolated points.
Let $G,H<\textup{Homeo}(X)$ be isomorphic groups such that both the actions are \emph{locally dense}.
Then for each isomorphism $\nu:G\to H$ there is a homeomorphism $\phi:X\to X$ such that 
$\nu(f)=\phi^{-1} f \phi$ for each $f\in G$.
\end{thm}

We remark that in \cite{KKL} it was demonstrated that for the case of groups of orientation preserving homeomorphisms of $\mathbf{R}$ (or an open interval),
locally dense and locally minimal are equivalent.
(Of course, it is clear that locally minimal actions are locally dense.)
Therefore in our application of Rubin's theorem, we shall use the notion of locally minimal.

For the basics on Thompson's group $F$ we refer the reader to \cite{CFP},
and for the basics on Brown-Stein-Thompson groups and Bieri-Strebel groups we refer the reader to \cite{Stein} and \cite{BSt} respectively.
We shall only need standard facts about these groups.

Other notions that we shall need concern the theory of orbit equivalence relations.
Given a group action $G<\textup{Homeo}^+(\mathbf{R})$,
we consider the associated orbit equivalence relation $E\subset \mathbf{R}\times \mathbf{R}$
which is defined as the set of pairs $(x,y)$ with the property that there is a $g\in G$ such that $x\cdot g=y$.
For $x\in \mathbf{R}$, we denote by $[x]$ as the orbit of $x$ in $E$.
If the underlying group is countable, this is a borel equivalence relation with countable equivalence classes.
An equivalence relation is said to be \emph{finite} or \emph{countable} if the equivalence classes are respectively finite or countable.
A countable borel equivalence relation $E$ is said to be \emph{hyperfinite}, if it can be expressed as an increasing union of finite equivalence relations.

The following is folklore (See example $3.3$ in \cite{Moore}) and shall be required for the purposes of this article:

\begin{thm}
The action of $PSL_2(\mathbf{Z})$ on $\mathbf{R}\cup \{\infty\}$ by projective transformations produces a hyperfinite equivalence relation.
\end{thm}

In particular, the action of $F$ described in example \ref{F1} produces a hyperfinite equivalence relation,
since the associated orbit equivalence relation is the same as that of the above (upon restriction to $\mathbf{R}$).
Since the actions of $F$ in \ref{F1} and \ref{F2} are topologically conjugate, and since hyperfiniteness is preserved under topological conjugacy,
both actions produce hyperfinite equivalence relations.
Note that the fact that the actions are topologically conjugate is folklore, but it is also a consequence of Theorem \ref{rigidity}. 

We shall also use the notion of a $\mu$-amenable equivalence relation.
Consider a group action $G<\textup{Homeo}^+(\mathbf{R})$,
and the associated orbit equivalence relation $E\subset \mathbf{R}\times \mathbf{R}$.
Let $\mu$ be a borel measure on $\mathbf{R}$.
Then $E$ is said to be $\mu$-amenable if there is a $\mu$-measurable assignment $x\to \nu_x$ such that:
\begin{enumerate}
\item Each $\nu_x$ is a finitely additive probability measure defined on the orbit of $x$ satisfying $\nu_x([x])=1$.
\item If $(x,y)\in E$ then $\nu_x=\nu_y$. 
\end{enumerate}
By \emph{measurable}, we mean that if $A$ is any measurable subset of $X\times X$,
then $$x\to \nu_x(\{y\mid (x,y)\in A\})$$ is measurable.
We shall need the following result.
(See for instance Theorem $3.5$ in \cite{Moore}).

\begin{thm}
Let $E\subset \mathbf{R}\times \mathbf{R}$ be a hyperfinite equivalence relation.
Then for any borel measure $\mu$ on $\mathbf{R}$, $E$ is $\mu$-amenable.
\end{thm}

\section{Algebraic properties of groups that admit coherent actions}

The class of groups that admit coherent actions is a continuum family with a rich subgroup structure:
\begin{lem}\label{Continuum}
Let $H<\textup{Homeo}^+(\mathbf{R})$ be a group. There is a coherent group action $G<\textup{Homeo}^+(\mathbf{R})$
such that the underlying group $H$ is a subgroup of the underlying group $G$.
There are uncountably many isomorphism types of finitely generated groups that admit coherent actions on the real line (or a compact interval).
\end{lem}

\begin{proof}
Consider a faithful action of $H$ by homeomorphisms of the unit interval $[0,1]\subset \mathbf{R}$, so that each element fixes each point in the complement of this interval.
Let $G_1<\textup{Homeo}^+(\mathbf{R})$ be any coherent group action (if desired, the underlying group can be chosen to be finitely generated. For instance, take one of the actions of $F$ described in examples \ref{F1} and \ref{F2}.)
Then $\langle G_1, H\rangle$ is the required coherent action.
Recall that in \cite{KKL} it was shown that there are uncountably many isomorphism types of $2$-generated groups that admit a faithful action on the real line by homeomorphisms.
Hence the final conclusion of the Lemma is straightforward.
\end{proof}

The main goal of this section is to prove Theorem \ref{structure}.\\

\begin{proof}[Proof of theorem \ref{structure} part (1)]
Let $f,g\in G$ be the elements as prescribed by parts $(3),(4)$ of the definition of coherent.
Moreover, by replacing $f,g$ with their inverses if needed, assume that $f,g$ are increasing on neighborhoods of $\infty,-\infty$ respectively.
Let $$r_1=sup\{r\in \mathbf{R}\mid r\cdot f=r\}\qquad r_2=inf(Supp(f))$$
Clearly, $r_1,r_2$ are fixed by $f$, $f$ is increasing on $(r_1,\infty)$ and the identity on $(-\infty,r_2)$.

Let $$p_1=inf\{r\in \mathbf{R}\mid r\cdot g=r\}$$
Using minimality, we find an $h\in G$ such that $p_1\cdot h\in (r_1,\infty)$.
Consider the element $g_1=h^{-1} g h$.
Let $s=sup(Supp(g_1))$.
Note that $g_1$ is the identity on $(s,\infty)$.\\

{\bf Claim}: There are $m,n\in \mathbf{N}$ such that $\langle g_1^n, f^m\rangle \cong F$.\\

{\bf Proof}: There is an $n\in \mathbf{N}$ such that $r_2\cdot g_1^n>r_1$.
Let $r_2\cdot g_1^n=r_3$.
Moreover, there is an $m\in \mathbf{N}$ such that $r_3\cdot f^m>s$.
It follows that the group generated by $g_1^n,f^m$ satisfies the following relations:
\begin{enumerate}
\item $[h^{-1}f^m h,g_1^n]=1$.
\item $[h^{-2}f^m h^2,g_1^n]=1$.
\end{enumerate}
where $h=g_1^nf^m$.
Since the presentation $$\langle a,b \mid [(ab)^{-1}b (ab), a]=[(ab)^{-2} b (ab)^{2},a]=1\rangle$$ is a presentation for $F$, it means that the group generated by $g_1^n,f^m$
is a quotient of $F$.
Since the group is clearly non abelian, and since every proper quotient of $F$ is abelian, the group is isomorphic to $F$.
\end{proof} 

We shall need the following consequence of Theorem \ref{Higman}:

\begin{lem}\label{HigmanCor}
Let $G<\textup{Homeo}^+(\mathbf{R})$ be a group action such that:
\begin{enumerate}
\item For each element $g\in G$ there is a compact interval $I_g\subset \mathbf{R}$ such that $g$ fixes each point in $\mathbf{R}\setminus I_g$.
\item For each pair of compact intervals $J_1,J_2\subset \mathbf{R}$, there is an element $g\in G$ such that $J_1\cdot g\subset J_2$.
\end{enumerate}
Then $G'$ is simple.
\end{lem}

\begin{proof}
Let $\alpha, \beta, \gamma\in G\setminus \{1_{G}\}$.
Let $J_1$ be a compact interval that contains as a subset $S=D(\alpha)\cup D(\beta)$.
We choose a compact interval $J_2$ such that $\gamma(J_2)\cap J_2=\emptyset$.
Using the hypothesis, we find an element $g$ such that $J_1\cdot g\subset J_2$.
By a direct application of Theorem \ref{Higman}, we are done.
\end{proof}

Recall that $G_c$ is the subgroup of compactly supported elements of $G$.

\begin{lem}\label{compactlysupported}
Let $G<\textup{Homeo}^+(\mathbf{R})$ be a coherent group action.
Then for each $r_1,r_2\in \mathbf{R}$ such that $r_1<r_2$, there is an element $g\in G_c$ such that $r_1\cdot g>r_2$.
\end{lem}

\begin{proof}
Since $G$ is nonsolvable, and since the groups of germs at $\pm \infty$ are solvable, we can produce compactly supported non trivial elements in $G$
by considering long commutators of elements in $G$.
Let $f$ be such an element, and let $U\subset \mathbf{R}$ be a component of support of $f$.
For each $r\in [r_1,r_2]$, using minimality, we find an element $g_r\in G$ such that $r\cdot g_r\in U$.
It follows that $f_r=g_r f g_r^{-1}$ is a compactly supported element with a component of support $U_r$ that contains $r$.
These components of support form an open covering of $[r_1,r_2]$.
There is a finite subcovering, using compactness.
Let $f_{s_1},...,f_{s_n}$ be the respective elements with components of support $U_{s_1},...,U_{s_n}$ respectively covering $[r_1,r_2]$.
One can produce a word $W$ in these letters and their inverses such that $r_1\cdot W>r_2$.
Clearly, $W\in G_c$.  
\end{proof}

\begin{lem}\label{IntervalTransitivity}
Let $U,V$ be compact intervals in $\mathbf{R}$. Then there is a $g\in G_c$ such that $U\cdot g\subset V$.
\end{lem}

\begin{proof}
From coherence, let $f_1\in G$ be an element such that $f_1$ is the identity on $(-\infty, r_1)$ and has a component of support $(r_2,\infty)$ on which it is increasing.
Using minimality, let $f_2\in G$ be such that $r_2\cdot f_2\in int(V)$.
Consider the element $f_3= f_2^{-1} f_1 f_2$.
Note that $f_3$ has a component of support $(r_2\cdot f_2,\infty)$ and a trivial germ at $-\infty$.

Using Lemma \ref{compactlysupported}, we obtain an element $f_4\in G_c$ such that $inf(U)\cdot f_4>r_2\cdot f_2$.
In particular, $U\cdot f_4\subset (r_2\cdot f_2,\infty)$.
It follows that there is an $n\in \mathbf{N}$ such that $$(U\cdot f_4)\cdot f_3^{-n}\subset V$$
We shall produce an element of $G_c$ which agrees with $f_3^{-n}$ on $U\cdot f_4$.

Again using Lemma \ref{compactlysupported}, we find an element $f_5\in G_c$ such that $$inf(Supp(f_3))\cdot f_5> sup(U\cdot f_4)$$
Let $$g_1=f_3f_5^{-1} f_3^{-1} f_5$$
Clearly, $g_1\in G_c$ and moreover $g_1,f_3$ agree on the interval $(r_2\cdot f_2, sup(U\cdot f_4)]$.
In particular, $U\cdot g_1^{-n}\subset V$.
So the required element is $g=f_4 g_1^{-n}$.
\end{proof}

\begin{proof}[Proof of theorem \ref{structure} (2) and (3)]
Using Lemma \ref{IntervalTransitivity} and \ref{HigmanCor} we conclude that $G_c'$ is simple.
Since $G$ has solvable germs at $\pm \infty$, it follows that for some $n\in \mathbf{N}$,
$G^{(n)}\subseteq G_c$ and hence $G^{(n+1)}\subseteq G_c'$, and since the latter is simple, $G^{(n+1)}=G_c'$.
Since each group $G^{(1)},...,G^{(n+1)}$ has a trivial centraliser in $G$, it follows that every non trivial normal subgroup of $G$ must contain $G^{(n+1)}$ and hence the respective quotient must be solvable of degree at most $n+1$. 
\end{proof}

We end this section by providing an elementary general construction of coherent actions where the underlying group has solvable quotients of desired length of solvability.

\begin{lem}\label{solvability}
For each $n\in \mathbf{N}$, there is a coherent group action $G<\textup{Homeo}^+(\mathbf{R})$ such that the underlying group admits quotients of solvable length $n$.
\end{lem}

\begin{proof}
Let $H$ be a left orderable group of solvable length $n$.
Consider a dynamical realisation of $H$ as homeomorphisms of $[0,1]$,
and consider another dynamical realisation of $H$ as homeomorphisms of
$\mathbf{R}$ where the restriction of the action on each interval $[n,n+1]$ (for $n\in \mathbf{Z}$)
is the conjugate by the integer translation $x\to x+n$ of the dynamical realisation defined on $[0,1]$.
We call this group action $H_1<\textup{Homeo}^+(\mathbf{R})$ (where the underlying group $H_1\cong H$.)
Note that by construction, elements of $H_1$ commute with integer translations.

Now consider the group action $G<\textup{Homeo}^+(\mathbf{R})$ which is generated by $H_1$ together with 
$F<\textup{Homeo}^+(\mathbf{R})$ where $F$ is the coherent action of Thompson's group defined in example \ref{F1}.
Since the germs at $\pm \infty$ of this action of Thompson's group are integer translations,
the germs at $\pm \infty$ of $G$ are isomorphic to $\mathbf{Z}\times H$.
In particular, $G$ is a coherent action.
The homomorphism to the germ at $\infty$ (or $-\infty$) then provides the required quotient.
\end{proof}

\section{Rigidity of coherent actions}

The goal of this section is to prove Theorem \ref{rigidity}.
The proof of this shall involve an application of Rubin's theorem, and the main technical Proposition necessary is the following.

\begin{prop}\label{locallydense}
Let $I$ be a compact interval or $I=\mathbf{R}$.
Every coherent group action $G<\textup{Homeo}^+(I)$ is locally minimal on $int(I)$.
\end{prop}

For the rest of the section we assume that $G$ satisfies the hypothesis of the Proposition.
We shall only consider the case where $I=\mathbf{R}$.
The other case is similar.
Recall that since $G$ is not solvable, but the groups of germs of $G$ at $-\infty,+\infty$ are,
we can construct non trivial elements of $G$ with trivial germs at $\pm \infty$ using iterated commutators. 
In this section, we first prove a sequence of refinements of minimality for coherent actions,
building upon the ones we proved in the previous section.
Then we shall use these refinements to prove Proposition \ref{locallydense}.

\begin{lem}\label{trans}
Let $U,V$ be closed subintervals of $\mathbf{R}$ such that $sup(U)<inf(V)$.
For every $r\in \mathbf{R}$, there is an element $g\in G$ such that:
\begin{enumerate}
\item $g$ fixes each point in $U$.
\item $inf(V\cdot g)>r$. 
\end{enumerate}
\end{lem}

\begin{proof}
Let $f\in G$ be an element such that for some $r_1,r_2\in \mathbf{R}, r_1<r_2$, $f$ fixes each point in $(-\infty,r_1)$ and is increasing on $(r_2,\infty)$. 
Using Lemma \ref{IntervalTransitivity}, we find a $f_1\in G$ such that $$[r_1,r_2]\cdot f_1\subset (sup(U),inf(V))$$
Consider the element $g_1=f_1^{-1} f f_1$.
Clearly, there is an $n\in \mathbf{N}$ such that $$inf(V\cdot g_1^n)>r$$ and $$x\cdot g_1^n=x\qquad \forall x\in U$$
Therefore the required element is $g=g_1^n$.
\end{proof}

\begin{lem}\label{transepsilon}
Let $U,V,W$ be closed subintervals of $\mathbf{R}$ such that $sup(U)<inf(V)$.
For every $r\in \mathbf{R}$, there is an element $g\in G$ such that:
\begin{enumerate}
\item $U\cdot g\subset W$.
\item $V\cdot g\subset (r,\infty)$. 
\end{enumerate}
\end{lem}

\begin{proof}
Using Lemma \ref{IntervalTransitivity}, we find an element $f_1\in G$ such that $U\cdot f_1\subset W$.
Since $f_1$ is a homeomorphism, there is an $r_1\in \mathbf{R}$ such that $(r_1,\infty)\cdot f_1\subset (r,\infty)$.
Next, using Lemma \ref{trans}, we find an element $f_2$ that fixes $U$ pointwise and maps $V$ inside $(r_1,\infty)$.
So the required element is $g=f_2f_1$.
\end{proof}

\begin{lem}\label{finaltrans}
Let $U_1,U_2,V_1,V_2$ be closed intervals in $\mathbf{R}$ such that $$sup(U_1)<inf(U_2)\qquad sup(V_1)<inf(V_2)$$
Then there is a $g\in G$ such that $U_1\cdot g\subset V_1$ and $U_2\cdot g\subset V_2$.
\end{lem}

\begin{proof}
Using Lemma \ref{transepsilon}, we find a $g_1\in G$ such that:
\begin{enumerate}
\item $U_1\cdot g_1\subset V_1$.
\item $inf(U_2\cdot g_1)>sup(V_2)$. 
\end{enumerate}
Next, let $f\in G$ be an element which fixes an interval $(-\infty,r_1)$
pointwise and is increasing on an interval $(r_2,\infty)$.
Using lemma \ref{IntervalTransitivity}, we find an element $f_1\in G$ such that $[r_1,r_2]\cdot f_1\subset int(V_2)$.
Consider the element $f_2=f_1^{-1} f f_1$.
It follows that $f_2$ fixes $V_1$ point wise and there is an $n\in \mathbf{N}$ such that $U_2\cdot f_2^{-n}\subset V_2$.
Therefore the required element is $g=g_1 f_2^{-n}$.
\end{proof}

\begin{proof}[Proof of Proposition \ref{locallydense}]
Consider a triple $(U,V,x)$ where $U\subset V$, the sets $U,V$ are open intervals and $x\in V\setminus U$.
We assume that $sup(U)\leq x$, and the case when $x\leq inf(U)$ is similar.
Let $\delta>0$ be such that $(x,x+\delta)\subset V$.
Let $f\in G$ be a compactly supported element for which there are $$r_1,r_2,r_3,r_4\in \mathbf{R}\qquad  r_1<r_2<r_3<r_4$$ such that:
\begin{enumerate}
\item The support of $f$ is contained in $(r_1,r_4)$
\item The interval $(r_2,r_3)$ is a connected component of the support of $f$.
\end{enumerate} 
Then using Lemma \ref{finaltrans} we find $g\in G$ such that:
\begin{enumerate}
\item $[r_1,r_2]\cdot g\subset U$ 
\item $[r_3,r_4]\cdot g\subset (x,x+\delta)$.
\end{enumerate}
It follows that the element $f_1=g^{-1} f g$ has its support contained in $V$ and that there is an $n\in \mathbf{Z}$
such that $x\cdot f_1^n\in U$.
Hence proved.
\end{proof}

\begin{proof}[Proof of Theorem \ref{rigidity}]
The proof follows immediately from Proposition \ref{locallydense} and Theorem \ref{rubin}.
\end{proof}

\begin{remark}
Note that to prove Theorem \ref{rigidity},
we do not need the full strength of the definition of coherent actions.
Indeed reading through the proofs in this section, the reader will find that we have also shown the following statement.
\end{remark}

\begin{cor}
Let $G,H<\textup{Homeo}^+(\mathbf{R})$ be group actions such that:
\begin{enumerate}
\item Both actions are minimal.
\item Both actions contain non identity elements with supports contained in compact intervals.
\item There exist $g\in G, h\in H$ such that $g,h$ have trivial germs at $-\infty$ and there is an interval $(r,\infty)$
on which both $g,h$ are increasing.
\item $G\cong H$.
\end{enumerate}
Then for each isomorphism $\nu:G\to H$ there is a homeomorphism $\phi:\mathbf{R}\to \mathbf{R}$ such that $\nu(f)=\phi^{-1} f \phi$ for each $f\in G$.
\end{cor}

\section{Combinatorially finite group actions}

The goal of this section to develop some tools which shall be at the core of the applications.
We shall work with two important notions of \emph{combinatorially finite} and \emph{weakly coherent} actions,
which we now define.

A group action $G<\textup{Homeo}^+(I)$ (where $I=\mathbf{R}$ or $I$ is a compact interval)
is said to be \emph{combinatorially finite} if the following holds:
\begin{enumerate}
\item The underlying group is finitely generated.
\item Every element has finitely many components of support.
\item The groups of germs at each point $x\in I\cup \{sup(I),inf(I)\}$ are all abelian.
\end{enumerate}

A combinatorially finite action $G<\textup{Homeo}^+(I)$ is said to be \emph{weakly coherent} if the following holds.
Let $n$ be the maximum of the abelian rank of the groups of germs of $G$ at $inf(I),sup(I)$.
Then there is some coherent action $H<\textup{Homeo}^+(I)$ such that:
\begin{enumerate}
\item The groups of germs of $H$ at $inf(I),sup(I)$ contain abelian subgroups of rank $n$
\item $H\cong G$.
\end{enumerate}
Note that it is in fact true that if $G<\textup{Homeo}^+(I)$ is weakly coherent, and if $K<\textup{Homeo}^+(I)$
is any coherent group such that $K\cong G$, then $K$ satisfies the above.
This is an immediate consequence of Theorem \ref{rigidity}.
Weak coherence shall be a useful property for our applications, and in particular shall imply coherence in certain circumstances.

\begin{prop}\label{exceptionalcombfinite}
Let $G<\textup{Homeo}^+(I)$ be a combinatorially finite action such that:
\begin{enumerate}
\item The action is weakly coherent.
\item $int(I)$ does not contain a proper subinterval which is $G$-invariant.
\end{enumerate}
Then either the action is minimal, or it admits a unique exceptional, minimal, invariant set $\Gamma<I$. 
Moreover, the restriction of $G$ to $\Gamma$ is faithful.
\end{prop}

In the hypothesis of the above proposition when we say that the underlying group $G$ admits a coherent group action,
we mean that there is a coherent action $H<\textup{Homeo}^+(I)$
such that the underlying group $H$ is isomorphic to $G$.
However, the prescribed action of $G$ will not be coherent.

\begin{prop}\label{combfinite}
Let $G<\textup{Homeo}^+(I)$ be a combinatorially finite action satisfying the following:
\begin{enumerate}
\item The action is weakly coherent.
\item The action is minimal on $int(I)$.
\end{enumerate}
Then the given action is coherent.
\end{prop}

The two Propositions combine nicely to provide us with the following useful corollary.

\begin{cor}\label{maincor}
Let $G<\textup{Homeo}^+(I)$ be a combinatorially finite action such that:
\begin{enumerate}
\item The action is weakly coherent.
\item $I$ does not contain a proper subinterval which is $G$-invariant.
\end{enumerate}
Then the action is semiconjugate to a combinatorially finite and coherent action $H<\textup{Homeo}^+(I)$
such that the semiconjugacy induces an isomorphism between the underlying groups $G,H$.
\end{cor}

\begin{proof}
If the action is minimal on $int(I)$, then using Proposition \ref{combfinite} we conclude that this action is coherent.
Assume that the action is not minimal on $int(I)$.
By Proposition \ref{exceptionalcombfinite}, there is a unique exceptional, minimal, invariant set $\Gamma<I$.
Moreover, the restriction of $G$ to $\Gamma$ is faithful.
The desired semiconjugacy is then obtained by collapsing to a point the closure of each connected component of the complement of $\Gamma$ in $I$.
(See section $2.1.2$ in \cite{Navas}, for instance.)
It is easy to see that this new action is also combinatorially finite.
Moreover, the abelian rank of the group of germs at $inf(I),sup(I)$ is less than or equal to the rank of these groups pre-semiconjugation.
So the resulting action post-semiconjugation is also weakly coherent.
Moreover, it is minimal on $int(I)$, thanks to Proposition \ref{combfinite} we conclude that this action is coherent.
\end{proof}

The proofs of both Proposition \ref{combfinite} and \ref{exceptionalcombfinite} will follow from an application of the following:

\begin{lem}\label{elements}
Let $G<\textup{Homeo}^+(I)$ be a combinatorially finite action satisfying that:
\begin{enumerate}
\item The action is weakly coherent.
\item $I$ does not contain a proper subinterval which is $G$-invariant.
\end{enumerate}
Then the following holds:
\begin{enumerate}
\item There is an element that has a trivial germ at $inf(I)$ and is increasing on an interval of the form $(r,sup(I))$ for $r\in int(I)$.
\item There is an element that has a trivial germ at $sup(I)$ and is increasing on an interval of the form $(inf(I),r)$ for $r\in int(I)$.
\end{enumerate}
\end{lem}

We now supply the proofs of \ref{combfinite} and \ref{exceptionalcombfinite} using Lemma \ref{elements}.

\begin{proof}[Proof of Proposition \ref{combfinite}]
The action is assumed to be minimal on $int(I)$, and the groups of germs at $inf(I),sup(I)$ are abelian since the action is combinatorially finite.
The existence of the required elements follows immediately from Lemma \ref{elements}.
We conclude that the action is coherent.
\end{proof}

\begin{proof}[Proof of Proposition \ref{exceptionalcombfinite}]
If the action is minimal on $int(I)$, then we conclude directly that the action is coherent from Lemma \ref{elements}.
Assume that the action is not minimal on $int(I)$. 
Following Lemma \ref{elements}, let $f_1\in G$ be the element that has a trivial germ at $inf(I)$ and is increasing on an interval of the form $(r_1,sup(I))$ for $r_1\in int(I)$.
For any $x\in int(I)$, there is an element $g\in G$ such that $$x\cdot g\in (r_1,sup(I))$$
It follows that $$(x\cdot g)\cdot f^{-n}\to r_1$$
Therefore, every orbit in $int(I)$ accumulates to $r_1$.

Hence by Lemma \ref{minimalinvariant} this means there is a unique exceptional minimal invariant set $\Gamma\subset I$.
It suffices to show that the restriction of $G$ to $\Gamma$ is faithful.
Since the underlying group $G\in \mathcal{C}$, i.e. the underlying group admits some coherent action, we know that every proper quotient of $G$ is solvable.
So it suffices to show that the restriction of $G$ to $\Gamma$ contains a copy of Thompson's group $F$.
This will imply that the restriction is non solvable, and hence isomorphic to the original group.

Let $f_1$ be as above, and let $f_2$ be the element that has a trivial germ at $sup(I)$ and is increasing on an interval of the form $(inf(I), r_2)$ for some $r_2\in int(I)$.
Let $g\in G$ be such that $r_2\cdot g>r_1$.
It is possible to find such a $g$ since the action on $I$ has no proper invariant subinterval.
Let $f_3=g^{-1} f_2 g$.
Note that $(r_1,r_2\cdot g)\cap \Gamma\neq \emptyset$.
In fact, $r_1,r_2\cdot g$ are accumulation points of $(r_1,r_2\cdot g)\cap\Gamma$.

Now by the same argument as in the proof of theorem \ref{structure} part $(1)$, there are $m,n\in \mathbf{N}$ such that $\langle f_3^m,f_1^n\rangle$ is isomorphic to Thompson's group $F$.
By our assumptions above, it is easy to see that the restrictions of $f_1^n,f_3^m$ to $\Gamma$ do not commute, yet satisfy the relations of Thompson's group $F$.
Since every proper quotient of $F$ is abelian, it follows that restriction of $\langle f_3^m,f_1^n\rangle$ to $\Gamma$ is a faithful action of $F$.
This proves our claim that the restriction of $G$ to $\Gamma$ is non solvable and hence faithful.
\end{proof}

Our goal in the rest of the section is to prove Lemma \ref{elements}.
For notational convenience, we restrict ourselves to the case when $I=\mathbf{R}$, and the case where $I$ is a compact interval is dealt with similarly.
Let $G$ be as in the statement of the Lemma.
We fix a coherent action $G_1<\textup{Homeo}^+(\mathbf{R})$,
such that $G_1\cong G$.
We fix an isomorphism $\phi:G_1\to G$.
The notation we just defined above will be fixed throughout the rest of the section.

An element $f\in \textup{Homeo}^+(\mathbf{R})$ is said to be of:
\begin{enumerate}
\item \emph{type A}, if it has a trivial germ at $\infty$ and it does not fix any point in an interval of the form $(-\infty,r)$.
\item \emph{type B}, if it has a trivial germ at $-\infty$ and it does not fix any point in an interval of the form $(s,\infty)$.
\item \emph{type C}, if it does not fix any point in intervals of the form $(-\infty,r)$ and $(s,\infty)$ for some $r<s$, and fixes the points $r,s$.
\item \emph{fully supported} if it does not fix any points in $\mathbf{R}$.
\end{enumerate}

\begin{lem}\label{fullysupported}
There is an element $g\in G_1$ which is fully supported.
\end{lem}

\begin{proof}
Consider the elements $f,g_1$ we construct in the proof of theorem \ref{structure} part (1).
Note that the construction holds for any coherent action, so it does for $G_1$ in particular.
There is an $n\in \mathbf{N}$ such that $f^ng_1^n$ is increasing on $\mathbf{R}$.
\end{proof}

\begin{lem}\label{fullysupportedimage}
Let $g\in G_1$ be an element which is fully supported.
Then the element $\phi(g)$ is either of type C or is fully supported.
\end{lem}

\begin{proof}
Since the group action $G_1$ is coherent, we know that for some $n\in \mathbf{N}$, the elements of $G_1^{(n)}$ have their supports contained in compact intervals.
(Also since the underlying group is non solvable, $G_1^{(n)}$ is non trivial.)
Note that by elementary arguments it follows that $g$ does not commute with any element of $G_1^{(n)}$.

Assume that $\phi(g)$ has a trivial germ at $\infty$ (a similar argument works for the case of $-\infty$).
Then there is an interval $(r,\infty)$ such that any element of $G$ whose support is contained in $(r,\infty)$
commutes with $\phi(g)$.

Using our hypothesis on germs, it is straightforward to conclude that the support of each element of $G^{(n)}$ lies in some compact interval of $\mathbf{R}$. 
Indeed, we can construct elements of $G^{(n)}$ whose support is contained in $(r,\infty)$.
For instance, take any non trivial element $h\in G^{(n)}$ and conjugate it by an appropriate element of $G$ so that the support of the conjugate is contained in $(r,\infty)$.
(This uses the fact that there is no $G$-invariant proper subinterval of $\mathbf{R}$.)

Since the conjugate is also in $G^{(n)}$, this contradicts our previous observation.
Therefore the germs of $\phi(g)$ at $\pm \infty$ must both be non trivial.
\end{proof}

\begin{lem}\label{lemtype}
The following holds:
\begin{enumerate}
\item Let $f\in \textup{Homeo}^+(\mathbf{R})$ be an element of type A, B or C, and let $g\in \textup{Homeo}^+(\mathbf{R})$.
Then the element $g^{-1} f g$ is of the same type as $f$.
\item Let $f\in \textup{Homeo}^+(\mathbf{R})$ be an element of type A, B or C, and let $g\in \textup{Homeo}^+(\mathbf{R})$ be a fully supported element.
Then for any $n\in \mathbf{Z}\setminus \{0\}$, the homeomorphisms $g^{-n} f g^n$ and $f$ do not commute.
\end{enumerate}
\end{lem}

\begin{proof}
The proof of $(1)$ is elementary and left to the reader.
We shall prove $(2)$.

We prove this for an element $f$ of type A. The proofs for the other types shall be similar.
Let $(-\infty,r)$ be an interval such that $f$ does not fix a point in $(-\infty,r)$ and fixes $r$.
By replacing $f$ with its inverse if needed, we can also assume that $f$ is decreasing on $(-\infty,r)$.

Assume without loss of generality that the fully supported element is increasing on $\mathbf{R}$.
Given $n\in \mathbf{N}\setminus \{0\}$, we will show that $f$ and $f_1=g^{-n} f g^{n}$ do not commute.
It will follow immediately that for any $n\in \mathbf{Z}\setminus \{0\}$, $f$ and $g^{-n} f g^{n}$ do not commute.

Our claim follows from observing that $$(r\cdot f)\cdot f_1=x\cdot f_1\neq (x\cdot f_1)\cdot f$$
\end{proof}

\begin{lem}\label{germinfinity}
Given any element $f\in G_1$ which is either of type A,B, or C,
the element $\phi(f)\in G$ is also of one of the types A,B,C or is fully supported.
\end{lem}

\begin{proof}
Without loss of generality assume that $f$ is of type A.
The proof for type B,C will be similar.
Assume by way of contradiction that $\phi(f)$ has its support contained in a compact interval.

Using Lemma \ref{fullysupported} we find an element $g\in G_1$ which does not fix a point in $\mathbf{R}$.
By Lemma \ref{fullysupportedimage}, $\phi(g)$ has a non trivial germ at $\infty$.
This together with the fact that $G$ is combinatorially finite means that there is an interval $(r,\infty)$ on which either $\phi(g)$ or $\phi(g^{-1})$ is increasing.
(We assume the former for the sake of notational convenience).

We first find an $h\in G$ such that $$k=h^{-1} \phi(f) h$$ has its support contained in $(r,\infty)$.
(This uses the fact that there is no $G$-invariant proper closed subinterval of $\mathbf{R}$.)
Since $\phi(f)$ fixes each point outside a compact interval, there is an $n\in \mathbf{N}$ such that $$[\phi(g)^{-n} k \phi(g)^n,k]=1$$ since they have disjoint support.

Next, note that $$\phi^{-1}(k)=\phi^{-1}(h^{-1}) f \phi^{-1}(h)$$ is of also of type A thanks to Lemma \ref{lemtype}.
It also follows from Lemma \ref{lemtype} that for each $n\in \mathbf{Z}\setminus \{0\}$, the elements $$g^{-n} \phi^{-1}(k) g^n\text{ and }\phi^{-1}(k)$$ do not commute.
This is a contradiction since $\phi$ is an isomorphism.

This means that our original assumption must be false and that $\phi(f)$ has a non trivial germ at $\infty$ or $-\infty$ or both.
Since $\phi(f)$ has finitely many components of support, we conclude that it must be of type A,B,C or fully supported.
\end{proof}

We remark that in the above Lemma, the type of $f$ and $\phi(f)$ may not be the same.

\begin{proof}[Proof of Proposition \ref{elements}]
We will show that there are elements $f_1,f_2\in G$ satisfying:
\begin{enumerate}
\item $f_1$ has a trivial germ at $-\infty$ and does not fix any point in an interval of the form $(r,\infty)$.
\item $f_2$ has a trivial germ at $\infty$ and does not fix any point in an interval of the form $(-\infty,r)$.
\end{enumerate} 
In fact, we only show the existence of $f_1$. 
The proof of the existence of $f_2$ shall be symmetric.

We know that the action of $G_1$ is coherent, and the groups of germs at $\pm \infty$ contain an abelian subgroup of rank $n$. 
We find elements $g_1,..,g_n$ and $h_1,...,h_n$ in $G_1$ such that:
\begin{enumerate}
\item The germs of $g_1,...,g_n$ at $-\infty$ generate an abelian group of rank $n$.
\item The germs of $h_1,...,h_n$ at $\infty$ generate an abelian group of rank $n$.
\end{enumerate}
Moreover, let $g,h\in G_1$ be elements such that:
\begin{enumerate}
\item $g$ has a trivial germ at $-\infty$ and does not fix a point in an interval of the form $(r,\infty)$.
\item $h$ has a trivial germ at $\infty$ and does not fix a point in an interval of the form $(-\infty, s)$.
\end{enumerate}

The elements $g,g_1,...,g_n$ satisfy that for each $(a_0,a_1,...,a_n)\in \mathbf{Z}^{n+1}\setminus \{(0,...,0)\}$,
the element $$h_{a_0,...,a_n}=g^{a_0}g_1^{a_1}...g_n^{a_n}$$ either has a non trivial germ at $-\infty$ or $\infty$.
In particular, it is of one of the types A,B,C or fully supported.

Using the assumption that the abelian rank of the group of germs at $-\infty$ 
for $G$ is at most $n$, we have the following.
There is an element $$(a_0,a_1,...,a_n)\in \mathbf{Z}^{n+1}\setminus \{(0,...,0)\}$$ such that
the element
$$\phi(h_{a_0,...,a_n})=\phi(g)^{a_0}\phi(g_1)^{a_1}...\phi(g_n)^{a_n}$$
has a trivial germ at $-\infty$.

Using Lemma \ref{germinfinity}, we know that $\phi(h_{a_0,...,a_n})$ must have a non trivial germ at $\infty$.
Since the group $G$ is combinatorially finite, we conclude that $\phi(h_{a_0,...,a_n})$ does not fix a point in an interval of the form $(r,\infty)$.
Therefore, the required element is $f_1=\phi(h_{a_0,...,a_n})$.
The element $f_2$ can be found in a symmetric fashion replacing the roles of $g,g_1,...,g_n$ by $h,h_1,...,h_n$.
\end{proof}

\section{Proof of the main theorems}

In this section we shall prove Theorems \ref{main1}, \ref{main2} and \ref{main3}.

\begin{proof}[Proof of Theorem \ref{main1}]
By way of contradiction, assume that such a $G\in \mathcal{C}$ is a subgroup of the standard copy of $F$ in $\textup{PL}^+([0,1])$.
We denote the underlying group as $G$, and also denote the action as $G< F< \textup{PL}^+([0,1])$.
By our hypothesis, there is a coherent action $H<\textup{Homeo}^+([0,1])$ such that $G\cong H$ and $H$ has a group of germs at a point which is either non abelian or abelian rank greater than $1$.

Since from Theorem \ref{structure} $G$ satisfies that every proper quotient is solvable of some bounded length, $G$ does not embed in a direct product of its proper quotients.
Therefore, there is a closed interval $I\subset [0,1]$ which does not contain a proper $G$-invariant closed subinterval and such that the restriction of $G$ to $I$ is faithful.
Note that $G|_I$ is combinatorially finite and also weakly coherent, since the group of germs at $inf(I),sup(I)$ are isomorphic to $\mathbf{Z}$.
Using Corollary \ref{maincor} we furnish a combinatorially finite, coherent group action $$K<\textup{Homeo}^+(I)$$ such that $G\cong K$.
Note that the semiconjugation also provides that for $K$ the groups of germs at points have abelian rank one.
Using Theorem \ref{rigidity}, we conclude that this is topologically conjugate to $H$.
This is a contradiction, since the fact that groups of germs at points are $\mathbf{Z}$ is preserved under topological conjugacy.
\end{proof}

\begin{proof}[Proof of Theorem \ref{main2}]
In this proof $\mu$ refers to the Lebesgue measure on $\mathbf{R}$.
It suffices to consider the case of a finitely generated group, since if a countable subgroup of $\textup{Homeo}^+(\mathbf{R})$
produces a non $\mu$-amenable equivalence relation, there is a finitely generated subgroup which also produces a non $\mu$-amenable equivalence relation.
This is true since the equivalence relation is an increasing union of the equivalence relations produced by an increasing sequence of finitely generated subgroups.
And since if each equivalence relation in the sequence is $\mu$-amenable, the union is $\mu$-amenable.
(See \cite{KechrisMiller} or \cite{Moore} for the closure properties of amenable equivalence relations.)
If this finitely generated subgroup is shown not to embed in $F$, then the overgroup does not embed in $F$.

Let $G_1<\textup{Homeo}^+(\mathbf{R})$ be a coherent action of a finitely generated group satisfying that the associated orbit equivalence relation is non $\mu$-amenable.
Assume by way of contradiction that there is a subgroup $G$ of the standard copy of $F$ in $\textup{PL}^+([0,1])$ such that $G_1\cong G$.
Since $G$ satisfies that every proper quotient is solvable of some bounded length, it does not embed in a direct product of its quotients.
So we can find an interval $I\subset \mathbf{R}$, such that the action of $G$ restriction to $I$ is faithful and $I$ has no proper invariant closed subinterval.
Note that $G|_I$ is combinatorially finite and also weakly coherent, since the group of germs at $inf(I),sup(I)$ are isomorphic to $\mathbf{Z}$.

There are two cases:
\begin{enumerate}
\item The action of $G$ on $int(I)$ in minimal.
\item The action of $G$ on $int(I)$ admits an exceptional invariant set $\Gamma\subset int(I)$ such that $G|_{\Gamma}$ is minimal and faithful.
\end{enumerate}

In case $(1)$ it follows from Proposition \ref{combfinite} that this action is coherent. 
From Theorem \ref{rigidity} we know that this action is topologically conjugate to the coherent action $G$ that produces a nonamenable equivalence relation with respect to the Lebesgue measure.
This is a contradiction, since the equivalence relation of $G$ on $I$ is amenable with respect to the Lebesgue measure.

Now we treat case $(2)$. 
Since $G$ is a combinatorially finite and weakly coherent group action, using Corollary \ref{maincor}, we furnish a combinatorially finite, coherent group action $$H<\textup{Homeo}^+(I)$$ such that $G\cong H$.
Let $\Phi:I\to [0,1]$ be the semiconjugacy map such that $\Phi(\Gamma)=(0,1)$ where $\Gamma\subset I$ is the exceptional minimal invariant set for $G$.

Since the restriction of the equivalence relation to $I$ is also hyperfinite, it is $\mu$-amenable with respect to the natural pullback $\nu$ of the Lebesgue measure under $\Phi$
which assigns measure $1$ to $\Gamma$ in $I$.
Note that the restriction of $\Phi$ to $\Gamma$ is a measure preserving borel bijection outside a $\nu$-null set.
Moreover, outside this $\nu$-null set, $\Phi$ maps orbits to orbits.
It follows that if $E$ is the associated orbit equivalence relation of the action of $H$ on $I$, $E$ is amenable with respect to Lebesgue measure.

Since $G_1,H$ are topologically conjugate by Theorem \ref{rigidity}, we obtain a contradiction since one equivalence relation is nonamenable and the other is amenable with respect to the Lebesgue measure.
This means that our assumption that $G_1$ embeds in Thompson's group $F$ must be false.
\end{proof}

\begin{proof}[Proof of Theorem \ref{main3}]
By way of contradiction, let $G$ be a subgroup of the standard copy of $F$ in $\textup{PL}^+([0,1])$ such that $G$ is isomorphic to the given group $G_1$.
Since $G$ satisfies that every proper quotient is solvable of some bounded length, it does not embed in a direct product of its quotients.
Therefore, there is an closed interval $I\subset \mathbf{R}$ which does not contain a proper $G$-invariant closed subinterval and such that the restriction of $G$ to $I$ is faithful.
Again, $G$ is weakly coherent because the groups of germs at $inf(I),sup(I)$ are abelian.
Since $G$ is a combinatorially finite and weakly coherent group, using Corollary \ref{maincor}, we furnish a combinatorially finite, coherent group action $$H<\textup{Homeo}^+(I)$$ such that $G\cong H$.
Using Theorem \ref{rigidity}, we conclude that this is topologically conjugate to $G$.
This is a contradiction, since such a topological conjugacy preserves the property that each element has finitely many components of support.
\end{proof}

\begin{proof}[Proof of Corollary \ref{maincor1}]
This is an immediate consequence of Theorem \ref{main1}. 
\end{proof}

\begin{proof}[Proof of Corollary \ref{maincor2}]
The proof goes along the lines of the proof of Theorem \ref{main1}. 
Note that the full strength of the notion of weak coherence is needed here.
\end{proof}

\begin{proof}[Proof of Corollary \ref{maincor3}]
The proof goes along the same lines as the proof of Theorem \ref{main1}.
The only difference is that at the very end of the proof, the contradiction obtained arises from the existence of the element with infinitely many components of support.
\end{proof}

\end{document}